\documentclass[]{article}
\usepackage{dsfont}
\usepackage[cp1250]{inputenc}
\usepackage[T1]{fontenc}
\usepackage{theorem}
\usepackage{graphicx}
\usepackage[fleqn]{amsmath}
\usepackage{amsfonts}
\usepackage{amssymb}
\usepackage{bbm}
\usepackage[all]{xy}
\usepackage[hmargin=2in,vmargin=1.5in]{geometry}
\usepackage{color}
\usepackage{indentfirst}
\usepackage[numbers]{natbib}    

\newtheorem{definition}{Definition} [section]
\newtheorem{theorem}[definition]{Theorem}

\newtheorem{lemma}[definition]{Lemma}
\newtheorem{remark}[definition]{Remark}

\renewcommand{\P}{\mathbb{P}}

\newenvironment{proof}{\par\noindent {\bf Proof.}}
		{\begin{flushright} \vspace*{-1mm} \mbox{$\Box$} \end{flushright}}

\begin{document}
\title{Uniform random posets
				\footnotetext{This research was partially supported by Polish National Science Center - NCN, decision number 2013/09/B/ST6/02258 (OPUS 5).\\
				}
}

\author{
	Patryk Kozie{\l} \and Ma{\l}gorzata Sulkowska
}
			



\maketitle
\rm
\small
\date{}


\begin{abstract} We propose a simple algorithm generating labelled posets of given size according to the almost uniform distribution. By ``almost uniform'' we understand that the distribution of generated posets converges in total variation to the uniform distribution. Our method is based on a Markov chain generating directed acyclic graphs.
\end{abstract}

\vspace{0.2 cm}

{\bf Key words:} poset, random generator, uniform distribution, DAG, Markov chain

\vspace{0.2 cm}


\section{Introduction}
Partially ordered sets (or, in brief, posets) are widely investiagted mathematical structures. A poset consists of a set and binary relation which is reflexive, antisymmetric and transitive. These features are common for many structures modelling real-life phenomena. In general, thanks to transitivity, posets reflect the concept of ordering, they model well arrangements in which a pair of elements is either not comparable or one element precedes the other, e.g. logical task ordering, preferences of people, information flow through the network etc. 

A lot of research on posets has been done up to now (consult \cite{trotter_1} by Trotter or the chapter ``Partially Ordered Sets'' also by Trotter in \cite{trotter_2}). But surprisingly, some basic features are still not discovered, e.g. the exact number of posets on $n$ elements is known only for $n \leq 18$ in labelled case and for $n \leq 16$ in unlabelled case (the case $n=16$ was solved by McKay and Brinkmann in \cite{upto16}). Asymptotically, it is known that the logarithm of the number of posets on $n$ elements is $n^2/4 + 3n/2 + O(\log{n})$ (see \cite{kleitman}).

In this paper we will work with random posets. Random structures appeared in graph theory by Erd\"os and R\'enyi \cite{erdos}. Random graphs have received a lot of attention, standard model $G(n,p)$ is nowadays actually commonplace. Random orders however were not investigated this much, partly due to the fact that their transitivity forbids the independent choice of related pairs, which complicates significantly the process of random generation.

This paper describes a method which generates posets of given size according to the almost uniform distribution (the expression ``almost uniform'' will be explained formally in Section \ref{sec_uni_opos}). Of course, due to the fact that the number of posets grows so rapidly with the number of elements, it is impossible to generate all of them and just sample uniformly. P. Winkler in \cite{winkler_1} and \cite{winkler_2} proposed a method for generating random posets of bounded dimension by taking the intersection of randomly and independently chosen linear orders. As he says ``While our model lacks the flexibility and power of random graphs and does not weight orders uniformly, it admits a variety of approaches and may yet prove useful''. M. Albert and A. Frieze propose in \cite{frieze} generating random posets by taking a random labelled graph, directing the edges towards greater vertices and considering their transitive closures. They discuss some properties (width, dimension, first-order properties) of such obtained posets.

Our method weights posets almost uniformly. We believe that this generator will prove useful in many applications. The obvious one is generating random networks. Here one could mention e.g. activity networks for project scheduling problems or manufacturing (some reasonable network generator is always needed to test proposed solution methods, consult \cite{schedule_1} or \cite{schedule_2}). The other example is modelling the information flow through decentralized type of networks (such as ad hoc networks), where each node takes part in routing by forwarding data to the other nodes. From the fundamental research perspective almost uniform poset generator may be very helpful in investigating some average properties of posets, estimate some poset statistics or test conjectures about those structures.


One should mention that a couple of years ago a powerful method (called Boltzmann sampler) for generating combinatorial objects from desired distribution has been proposed by P. Duchon et al. and is still being developed (see \cite{boltz_1}, \cite{boltz_2} and \cite{boltz_3}). It applies to objects such as permutations, graphs, integer partitions, necklaces, ect. Nevertheless this method does not cope with objects being transitive, thus can not be applied to random poset generation.

The paper is organized as follows. Section 2 contains basic definitions and notation. In Section 3 we introduce the equivalence relation on the family of directed acyclic graphs, which we use later in the process of poset generation. In Section 4 we describe a method for almost uniform poset generation. Section 5 is dedicated to simulations.

\section{Basic definitions} 
\label{sec_def}

A {\it partially ordered set} or, in brief, {\it poset} is a pair $(X, R)$, where $X$ is a set and $R$ a reflexive, antisymmetric and transitive binary relation on $X$. For $a,b \in X$ we write $a \leq b$ whenever $(a,b) \in R$. We write $a < b$ if $a \leq b$ and $a \neq b$. We say that $a$ and $b$ are comparable if either $a \leq b$ or $b \leq a$. The cardinality of a poset is understood simply as a cardinality of $X$. Whenever $a < b$ and there is no $c$ such that $a<c$ and $c<b$, we say that $b$ {\it covers} $a$. If a poset is finite (throughout this paper we will consider only finite posets with $|X| = n$) it can be represented graphically by a {\it Hasse diagram}. In a Hasse diagram elements of $X$ are represented as vertices in the plane and a directed edge goes from $a$ to $b$ whenever $b$ covers $a$ (see Figure \ref{transCl}).

A {\it directed graph} $G$ is a pair $(V,E)$, where $V$ is a set of vertices and $E$ is a set of edges, i.e. ordered pairs of elements from $V$. A {\it size} of $G$ is understood simply as the cardinality of $V$. A {\it DAG} (which stands for {\it directed acyclic graph}) is a directed graph with no directed cycles. For $v, w \in V$ we say that $v$ is {\it reachable} from $w$ if there exists a directed path from $w$ to $v$ in $G$. A {\it transitive closure} of a DAG $G = (V,E)$ is a graph $\overline{G} = (V,E \cup F)$, where $F = \{(v,w): w {\rm~is~reachable~from~} v {\rm~in~} G\}$. A {\it transitive reduction} of a DAG $G = (V,E)$ is a graph $\underline{G} = (V, D)$, where $D \subseteq E$ and $D$ is the smallest subset of edges from $E$ which preserves reachability relation from $G$. Both, a transitive closure and a transitive reduction of a finite DAG is unique (see Figure \ref{transCl}).

\begin{figure}[!ht]
\centering
	 	\includegraphics[width=0.5\textwidth]{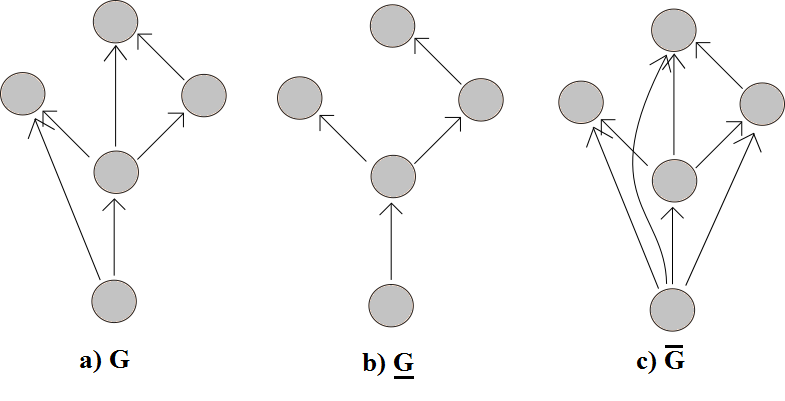}
		\caption{The example of transitive reduction and transitive closure of DAG $G$.}
		\label{transCl}
\end{figure}

Throughout this paper we will also work on a discrete-time Markov chain with a finite state space ${\cal{S}} = \{s_1, s_2, \ldots, s_N\}$. Such a {\it Markov chain} is a sequence of random variables $\{X_0, X_1, \ldots\}$, with each $X_i \in S$, following so-called Markov property: $\P[X_{t+1} = y | X_t = x_t, X_{t-1} = x_{t-1}, \ldots, X_0 = x_0] = \P[X_{t+1} = y | X_t = x_t]$. If $X_t = x$ we say that a chain is in state $x$ at time $t$. The value $\P[X_{t+1} = y | X_t = x]$ is called a {\it transition probability} and will be denoted by $p_{x,y}$. Thus a Markov chain can be described by {\it transition matrix} $P = [p_{x,y}]_{|S| \times |S|}$. An {\it initial distribution} $\mu^{(0)}$ is a row vector given by $\mu^{(0)} = (\mu^{(0)}_1, \mu^{(0)}_2, \ldots, \mu^{(0)}_N) = (\P[X_0 = s_1], \P[X_0 = s_2], \ldots, \P[X_0 = s_N])$. Similarly, $\mu^{(1)}, \mu^{(2)}, \ldots$ denote the distributions of the Markov chain at times $1,2,\ldots$. We have $\mu^{(k)} = \mu^{(0)} P^k$. A {\it stationary distribution} is a row vector $\pi = (\pi_1, \pi_2, \ldots, \pi_N)$ such that $\sum_{i=1}^{N} \pi_i = 1$ and $\pi P = \pi$. A state $x$ {\it communicates} with a state $y$ (we write $x \rightarrow y$) if there exists $k$ such that $\P[X_{m+k} = y|X_m = x] > 0$. States $x$ and $y$ {\it intercommunicate} if $x \rightarrow y$ and $y \rightarrow x$. A Markov chain is {\it irreducible} if every pair of distinct states intercommunicate. A {\it period} of a state $x$ is defined as $gcd\{k \geq 1: p_{x,x}^k > 0\}$. We say that a state is {\it aperiodic} if its period equals $1$. It is known that any Markov chain which is aperiodic and irreducible (which is called {\it ergodic}) has a unique stationary distribution (consult \cite{feller}).

We will also use a notion of a distance measure for probability distributions. For two probability distributions on ${\cal{S}} =\{s_1, s_2, \ldots, s_N\}$, $\nu = (\nu_1, \nu_2, \ldots, \nu_N)$ and $\xi = (\xi_1, \xi_2, \ldots, \xi_N)$ a {\it total variation distance} is given by $d_{TV}(\nu, \xi) = 1/2 \sum_{i=1}^{N}|\nu_i - \xi_i|$. We say that $\mu^{(m)}$ {\it converges to $\mu$ in total variation} if $\lim_{m \rightarrow \infty} d_{TV}(\mu^{(m)}, \mu) = 0$. Any irreducible and aperiodic (thus ergodic) Markov chain converges to its stationary distribution in total variation (consult \cite{levin}). We will be interested in how fast a Markov chain converges to its stationary distribution, thus we introduce also the notion of a {\it mixing time}. Let $d(t) = \max_{s \in S}\{d_{TV}(P^t(s,.),\pi)\}$, where $P^t(s,.)$ is the $s$th row of $P^t$, which is a distribution of $X_{t+1}$ under the condition $X_t = s$. If $d(t) < \varepsilon$, we say that Markov chain is {\it $\varepsilon$-close} to its stationary distribution. The {\it mixing time} of a Markov chain is given by $t_{mix}(\varepsilon) = \min\{t: d(t) < \varepsilon\}$. In practical considerations $\varepsilon$ is usually chosen as $1/(2 e) \approx 0.18$ or $1/4$ (see \cite{levin} and \cite{prasad}). 

\section{Equivalence relation on the family of DAGs} 

In our algorithm generating posets we will use the following equivalence relation on the family of DAGs of size $n$.

\begin{definition}
Let $G_1$ and $G_2$ be DAGs of size $n$. We write $G_1 \cong_n G_2$ if and only if $\overline{G_1} = \overline{G_2}$. 
\end{definition}
It is easy to check that $\cong_n$ is reflexive, symmetric and transitive, thus indeed it is the equivalence relation.

\begin{lemma}
\label{1-1}
The number of equivalence classes of $\cong_n$ equals the number of posets of size $n$.
\end{lemma}

\begin{proof}
Let $G=(V,E)$ be a DAG. Consider $\overline{G}=(V,F)$. Note that we can iterpret each edge $(v,w)$ of $\overline{G}$ as a pair which belongs to some binary relation on $V$. Since $\overline{G}$ is a transitive closure of DAG, the obtained relation will be reflexive, antisymmetric and transitive. Thus $(V,F)$ is a poset. On the other hand each poset can be easily transformed into DAG using the same interpretation of edges. The obtained DAG will form some transitive closure. Note that this operation gives us one-to-one correspondence between the family of posets of size $n$ and the family of equivalence classes of $\cong_n$. The conclusion follows.
\end{proof}
From now on by $[G]$ we denote the equivalence class of $G$ and by $|[G]|$ its cardinality.

\begin{lemma}
\label{l_eq_card}
Let $G$ be an arbitrary DAG with $k$ disconnected components\\
$G_1, G_2, \ldots, G_k$. The cardinality of the equivalence class of $G$ satisfies
$$ |[G]| = 2^{\sum_{i=1}^{k} (l_i-r_i)}, $$
where $l_i$ is the number of edges in $\overline{G_i}$ and $r_i$ is the number of edges in $\underline{G_i}$.
\end{lemma}

\begin{proof}
Consider an equivalence class of any $G$. Note that each graph from this class contains all the edges from $\underline{G} = (V,E)$ ($\underline{G}$ is here the smallest graph which preserves the proper reachability relation). $\overline{G} = (V,F)$ is in this class the richest graph preserving reachability relation. Thus every graph from this class is of the form $H=(V,E \cup D)$, where $D \subseteq F \setminus E$. $|F \setminus E| = \sum_{i=1}^{k} (l_i-r_i)$, thus $|[G]| = 2^{\sum_{i=1}^{k} (l_i-r_i)}$.
\end{proof}


\section{Almost uniform poset generation} 
\label{sec_uni_opos}

In this section we present an algorithm which samples labelled posets of given size almost uniformly. By ``almost uniformly'' we mean that the distribution of posets generated by our algorithm converges in total variation to the uniform distribution (see Theorem \ref{th_dtv_conv}).

Let ${\cal{P}} = \{P_1, P_2, \ldots, P_M\}$ be the family of all labelled posets on $n$ elements (the cardinality of ${\cal{P}}$ is denoted by $M$). Still, ${\cal{S}} = \{G_1, G_2, \ldots, G_N\}$ denotes the set of all labelled DAGs on $n$ vertices. The relation $\cong_n$ partitions ${\cal{S}}$ into $M$ equivalence classes $C_1, C_2, \ldots , C_M$ (recall Lemma \ref{1-1}) such that if $G \in C_i$ then $\overline{G}$ corresponds to $P_i$.

In our method we will modify the Markov chain that was introduced by Bousquet-M\'elou, Melan\c{c}on and Dutour in \cite{dags}. We start this section with describing the original Markov chain from \cite{dags}.

\subsection{Almost uniform DAG generation} 

Below we present a Markov chain introduced in \cite{dags}.
\begin{definition}[Markov chain $MC_n$]
Let ${\cal{S}}$ (the set of all labelled DAGs on $n$ vertices) be the state space of a Markov chain $MC_n = \{X_0, X_1, \ldots\}$. We start with an empty graph as $X_0$. 
At each step $t=1,2,\ldots$ we choose uniformly at random a directed edge $(i,j)$. Afterwards
\begin{enumerate}
\item If there exists an edge $(i,j)$ in $X_t$ then $X_{t+1} = X_t \setminus (i,j)$.
\item If there is no $(i,j)$ in $X_t$, then $X_{t+1} = X_t \cup (i,j)$ provided that the graph remains acyclic; otherwise $X_{t+1} = X_t$.
\end{enumerate}
\end{definition}
It is easy to verify that this Markov chain is irreducible and aperiodic and that its transition matrix is symmetric. It follows that its stationary distribution is uniform over the set of all labelled DAGs.

\subsection{DAG generator $MC^*_n$ with arbitrary stationary distribution}
\label{MC_star}
For the algorithm sampling posets almost uniformly we need a Markov chain with state space $\cal{S}$, but stationary distribution other than uniform. We aim for a stationary distribution $\pi = \{\frac{1}{M |[G_1]|}, \frac{1}{M |[G_2]|}, \ldots, \frac{1}{M |[G_N]|}\}$, which means that in stationary distribution each DAG corresponds with the probability being the reciprocal of the cardinality of its equivalence class multiplied by the normalization constant $1/M$.
In order to achieve it, we apply the Metropolis algorithm that transforms any irreducible Markov chain into Markov chain with a required stationary distribution. 

\begin{lemma}[\cite{mitz_upf}, Lemma 10.8]
\label{l_metro}
For a finite space $\cal{S}$ and the neighborhood structure $\{N(x): x \in {\cal{S}}\}$, let $L = \max_{x \in {\cal{S}}}|N(x)|$. Let $K$ be any number such that $K \geq L$. For all $x \in {\cal{S}}$, let $\pi_x > 0$ be the desired probability of state $x$ in the stationary distribution. Consider a Markov chain where
$$
  p_{x,y} = \left\{ \begin{array}{lclcl} (1/K \min\{1,\pi_y/\pi_x\}) & {\rm if} & x \neq y & {\rm and} & y \in N(x), \\																																								0													 & {\rm if} & x \neq y & {\rm and} & y \notin N(x), \\
  																			1- \sum_{x \neq y} p_{x,y} & {\rm if} & x=y. \end{array} \right.
$$
Then, if this chain is irreducible and aperiodic, the stationary distribution is given by the probabilities $\pi_x$.
\end{lemma}

Now, consider the following Markov chain.
\begin{definition}[Markov chain $MC^*_n$]
Let $\cal{S}$ be the state space of Markov chain $MC^*_n = \{X_0, X_1, \ldots\}$. Start with an empty graph at $X_0$. At each step $t=1,2,\ldots$ choose uniformly at random a directed edge $(i,j)$. Let $Y = X_{t} \cup (i,j)$ and $Z = X_t \setminus (i,j)$.
\begin{enumerate}
\item If there exists an edge $(i,j)$ in $X_t$ then with probability $\min\left\{1, \frac{|[X_t]|}{|[Z]|}\right\}$ set $X_{t+1} = Z$ and with probability $1-\min\left\{1, \frac{|[X_t]|}{|[Z]|}\right\}$ set $X_{t+1} = X_t$.
\item If there is no $(i,j)$ in $X_t$, then if $Y$ has a directed cycle, $X_{t+1}=X_t$; otherwise with probability $\min\left\{1, \frac{|[X_t]|}{|[Y]|}\right\}$ set $X_{t+1} = Y$ and with probability $1-\min\left\{1, \frac{|[X_t]|}{|[Y]|}\right\}$ set $X_{t+1} = X_t$.
\end{enumerate}
\end{definition}

\begin{remark}
Note that given any DAG we are able to calculate the cardinality of its equivalence class from Lemma \ref{l_eq_card}.
\end{remark}

\begin{lemma}
The stationary distribution of the Markov chain $MC^*_n$ is \\
$\pi = \left\{\frac{1}{M |[G_1]|}, \frac{1}{M |[G_2]|}, \ldots, \frac{1}{M |[G_N]|}\right\}$.
\end{lemma}
\begin{proof}
We define the neighborhood structure on $\cal{S}$ in a natural way. For $x, y \in {\cal{S}}$ we say that $x$ and $y$ are neighbors ($x \in N(y)$ and $y \in N(x)$) if they differ in exactly one edge (i.e., if we can obtain one from the other by adding or removing one edge). Then we can present the transition matrix of $MC^*_n$ as follows:
$$
  p_{x,y} = \left\{ \begin{array}{lclcl} (1/(n(n-1)) \min\{1,\pi_y/\pi_x\}) & {\rm if} & x \neq y & {\rm and} & y \in N(x), \\																																								0													 & {\rm if} & x \neq y & {\rm and} & y \notin N(x), \\
  																			1- \sum_{x \neq y} p_{x,y} & {\rm if} & x=y, \end{array} \right.
$$
where (as in Lemma \ref{l_metro}) for $x \in {\cal{S}}$ $\pi_x$ is the desired probability of the state $x$ in the stationary distribution ($\pi_x = 1/(M |[x]|)$). Note that $1/(n(n-1))$ is exactly the probability of drawing uniformly a directed edge $(i,j)$ and $\max_{x \in {\cal{S}}}|N(x)| = n(n-1)$ (consider e.g. an empty graph whose neighbors are all DAGs with only one edge). The chain is irreducible and aperiodic thus, by Lemma \ref{l_metro}, $\pi$ is the stationary distribution of $MC^*_n$.
\end{proof}
A DAG returned by $MC^*_n$ after performing $m$ steps will be denoted by $G^*_{n,m}$, i.e. $X_m = G^*_{n,m}$. 

\subsection{Almost uniform poset generator $PG_{n,m}$}

In this section we describe a method for almost uniform poset generation and discuss how close is  the resulting distribution to the uniform one.

\begin{definition}[Almost uniform poset generator $PG_{n,m}$]
Run a Markov chain $MC^*_n$ for $m$ steps. Return $\overline{G}^*_{n,m}$ as the desired poset.
\end{definition}

\begin{remark}
Recall that there exists a one-to-one correspondence between the family of transitive closures of DAGs of size $n$ and the family of posets on $n$ elements.
\end{remark}

\begin{theorem}
\label{th_dtv_conv}
Let $\xi^{(m)}$ be the distribution of the poset returned by $PG_{n,m}$. $\xi^{(m)}$ converges in total variation to the uniform distribution $\xi = \{1/M, \ldots, 1/M\}$, i.e. $\lim_{m \rightarrow \infty} d_{TV}(\xi^{(m)}, \xi) = 0$.
\end{theorem}
\begin{proof}
Recall that ${\cal{S}} = \{G_1, G_2, \ldots, G_N\}$ and ${\cal{P}} = \{P_1, P_2, \ldots, P_M\}$. For $i \in \{1, 2, \ldots, M\}$ let $C_i = \{j: \overline{G}_j = P_i\}$, i.e. $C_i$ is the set of indices of those DAGs which belong to the equivalence class of the poset $P_i$. Recall that $\pi = \{\frac{1}{M |[G_1]|}, \frac{1}{M |[G_2]|}, \ldots, \frac{1}{M |[G_N]|}\}$ is the stationary distribution of $MC^*_n$, thus $\pi^{(m)} = (\pi^{(m)}_1, \pi^{(m)}_2, \ldots, \pi^{(m)}_N)$ is the distribution of $MC^*_n$ in the $m$th step. Note that $\xi^{(m)}_i = \sum_{j \in C_i} \pi^{(m)}_j$ because $PG_{n,m}$ returns $P_i$ if and only if $MC^*_n$ returns in the $m$th step one of the DAGs from the equivalence class of $P_i$ (and these events are disjoint and independent). We have
\[
\begin{split}
d_{TV}(\xi^{(m)}, \xi) & = 1/2 \sum_{i=1}^{M}|\xi^{(m)}_i - 1/M| = 1/2 \sum_{i=1}^{M}\left|\sum_{j \in C_i} \pi^{(m)}_j - 1/M \right| \\
& = 1/2 \sum_{i=1}^{M}\left|\sum_{j \in C_i} \left(\pi^{(m)}_j - \frac{1}{|C_i| M}\right) \right| \\
& \leq 1/2 \sum_{i=1}^{M}\sum_{j \in C_i} \left|\pi^{(m)}_j - \frac{1}{|C_i| M}\right| \\
& = d_{TV}(\pi^{(m)}, \pi) \xrightarrow{m \rightarrow \infty} 0,
\end{split}
\]
where the inequality follows from the triangle inequality and the last line from the fact that $MC_n^*$ is irreducible and aperiodic, thus converges to its stationary distribution in total variation.
\end{proof}

We already know that the distribution of posets generated by $PG_{n,m}$ converges (in total variation) to the uniform one. However, we are also interested in how fast does it converge to know for how long should we run $MC^*_n$ in order to be acceptably close to the uniform distribution. The desired number of iterations can be described by the notion of mixing time (this is a very common approach while working with Markov chains, consult \cite{prasad} or \cite{levin}) of Markov chain which in this case seems to be difficult to calculate. However, one can state a conjecture drawn from the fact that the maximal distance between any two acyclic graphs is bounded by $n(n-1)$ - the path connecting them (for sure not always the optimal one) goes through the empty graph (one can delete all edges from the first graph and then add all the edges from the second graph). This observation suggests that performing the quadratic number of steps of $MC^*_n$ brings us acceptably close (e.g. $1/4$-close, recall the definition of mixing time from Section \ref{sec_def}) to the uniform distribution. The simulations we conducted (see Section \ref{sec_sim}) confirm the soundness of this conjecture.

\section{Results of simulations}
\label{sec_sim}

We have implemented and tested the generator $PG_{n,m}$ in Python 2.7.13 using version 7.6 of mathematical package SageMath. The experiments we have conducted suggest that running $MC^*_n$ for quadratic number of steps brings the distribution of posets close (even less then $0.1$-close for $n=4$ or $n=5$) to the uniform one.

\subsection{Time complexity of $PG_{n,m}$}

The overall time complexity of our algorithm depends on the complexity of two procedures. First is the method of checking if DAG is acyclic every time we try to add a new edge to $X_t$; the method we use has in worst case complexity ${\cal{O}}(n^2)$ . Second is counting the probability of moving to a new state of $MC_n^*$, which is combined with counting the number of edges in transitive closure and transitive reduction of some DAG (recall Lemma \ref{l_eq_card}); here the upper bound for number of computational operations is ${\cal{O}}(n^4)$. It may happen that we will have to perform both procedures in every step of the chain $MC_n^*$. Since we run $MC_n^*$ for quadratic number of steps, we obtain that the time complexity of our algorithm is ${\cal{O}}(n^6)$, thus polynomial in the number of elements of the poset.

\subsection{Empirical distribution of posets generated by $PG_{n,m}$}

In Figures \ref{hist_4} and \ref{hist_5} we present two histograms generated by running $PG_{n,m}$. We have generated $100~000$ samples of labelled posets of cardinality $4$ and the same number of samples of labelled posets of cardinality $5$. We ran $MC_n^*$ for quadratic number of steps. Thus below one can see the results of running $PG_{4,16}$ and $PG_{5,25}$.

\begin{figure}[!ht]
\centering
	 	\includegraphics[width=0.9\textwidth]{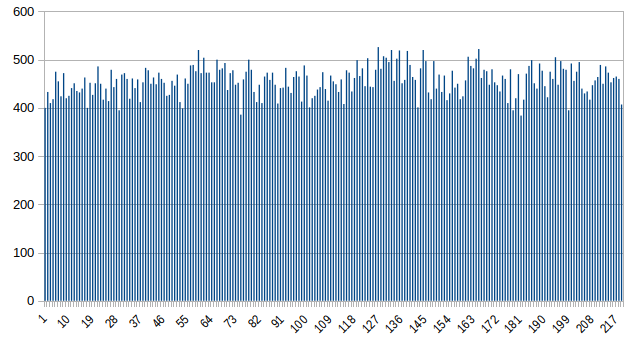}
		\caption{The histogram of 100~000 samples of labelled posets returned by $PG_{4,16}$. (There are $219$ posets of cardinality $4$; they were numbered in the order of appereance during the algorithm execution.)}
\label{hist_4}
\end{figure}

\begin{figure}[!ht]
\centering
	 	\includegraphics[width=0.9\textwidth]{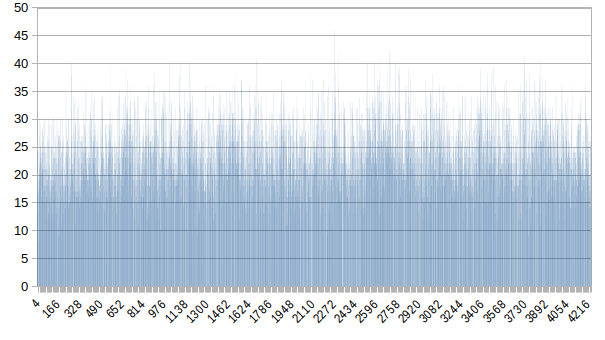}
		\caption{The histogram of 100~000 samples of labelled posets returned by $PG_{5,25}$. (There are $4231$ posets of cardinality $5$; they were numbered in the order of appereance during the algorithm execution.)}
\label{hist_5}
\end{figure}

We have calculated also the empirical total variation distances $\tilde{d}_{TV}$ as follows. For $n=4$ for $i \in \{1, 2, \ldots, 219\}$ let $z_i$ be the number of samples of $i$th category. Then $\tilde{d}_{TV} = 1/2 \sum_{i=1}^{219}|z_i/100~000 - 1/219|$ (and analogously for $n=5$). We obtained $\tilde{d}_{TV} \approx 0.026$ for $n=4$ and $\tilde{d}_{TV} \approx 0.092$ for $n=5$. Both values are significantly smaller than $1/4$ or $1/(2 e)$ - the constants chosen usually as acceptale in practical considerations while examining the mixing time (\cite{levin}, \cite{prasad}).

Presenting histograms for bigger $n$ would be rather illegible. Already for $n=6$ the number of labelled posets reaches $130~023$ while e.g. for $n=8$ about $4 \cdot 10^8$. Nevertheless, the algorithm samples easily the posets of bigger cardinalities. The example of a poset of size $21$ generated by $PG_{21,441}$ is given in Figure \ref{opos_21} (the exact number of labelled posets for $n>18$ is not known).

\begin{figure}[!ht]
\centering
	 	\includegraphics[width=0.4\textwidth]{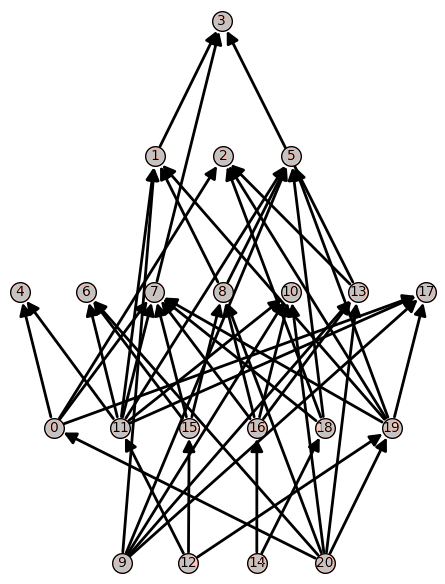}
		\caption{The poset on $21$ elements generated by $PG_{21,441}$.}
\label{opos_21}
\end{figure}

\section{Conclusion}
We have presented a method for generating labelled posets of given size from distribution close to uniform. We believe that it will find a number of applications, as we mentioned in the introduction. Another interesting step could be designing a method for generating unlabelled posets, wich seems to be pretty challenging. The first thing that comes to mind is that one will probably have to consider methods in which the isomorphism of two labelled posets is checked. This problem can be viewed as graph isomorphism problem. It was proved recently by L{\'{a}}szl{\'{o}} Babai \cite{DBLP:journals/corr/Babai15} that there exists a graph isomorphism test which runs in quasipolynomial time. This gives hope that the time complexity of the algorithm generating unlabelled posets would also stay reasonable.

\bibliographystyle{apa}
\bibliography{bib}


\end {document}